\documentclass{amsart}

\usepackage{amssymb,latexsym,amsfonts,amsmath}
\usepackage{graphicx,psfrag,epsfig,epsf}
\include{diagrams}
\usepackage{eso-pic}
\usepackage{graphicx}
\usepackage{color}
\usepackage{diagrams}
\usepackage[boxed]{algorithm2e}
\usepackage{tikz}
\usetikzlibrary{automata}
\usepackage{subfigure}
\usepackage{amsfonts}

\usepackage{amssymb,latexsym,amsfonts,amsmath}
\usepackage{graphicx}

\def\baselinestretch{1}

\newtheorem{theorem}{Theorem}[section]

\newtheorem{proposition}[theorem]{Proposition}

\theoremstyle{definition}
\newtheorem{definition}[theorem]{Definition}

\theoremstyle{remark}

\numberwithin{equation}{section}

\newcommand{\Post}{\operatorname{Post}}

\newcommand{\Id}{\mathrm{Id}}

\newcommand{\diag}{\mathrm{diag}}

\newcommand{\vecc}{\mathrm{vec}}
\newcommand{\SCC}{\mathrm{SCC}}
\newcommand{\Leaves}{\mathrm{Leaves}}
\newcommand{\temp}{\mathrm{temp}}
\newcommand{\DAG}{\mathrm{DAG}}
\newcommand{\Scc}{\mathrm{Scc}}
\newcommand{\card}{\mathrm{card}}
\newcommand{\sech}{\mathrm{sech}}
\newcommand{\tcomplex}{\mathrm{Tcomplex}}
\newcommand{\scomplex}{\mathrm{Scomplex}}

\begin{document}

\begin{abstract}                          
In this paper we propose symbolic models for networks of discrete--time nonlinear control systems. If each subsystem composing the network admits an incremental input--to--state stable Lyapunov function and if some small gain theorem--type conditions are satisfied, 
a network of symbolic models, each one associated with each subsystem composing the network, is proposed which is approximately bisimilar to the original network with any desired accuracy. Quantization parameters of the symbolic models are derived on the basis of the topological properties of the network.  
\end{abstract}

\title[Compositional Symbolic Models for Networks of Incrementally Stable Control Systems]{Compositional Symbolic Models for Networks of \\ Incrementally Stable Control Systems}

\thanks{The research leading to these results has been partially supported by the Center of Excellence DEWS and received funding from the European Union Seventh Framework Programme [FP7/2007-2013] under grant agreement n.257462 HYCON2 Network of excellence.}

\author[Giordano Pola, Pierdomenico Pepe and Maria D. Di Benedetto]{
Giordano Pola$^{1}$, Pierdomenico Pepe$^{1}$ and Maria D. Di Benedetto$^{1}$}
\address{$^{1}$
Department of Information Engineering, Computer Science and Mathematics, Center of Excellence DEWS,
University of L{'}Aquila, 67100 L{'}Aquila, Italy}
\email{ \{giordano.pola,pierdomenico.pepe,mariadomenica.dibenedetto\}@univaq.it}

\maketitle

\section{Introduction}

Symbolic models are abstract descriptions of control systems where any state corresponds to an aggregate of continuous states and any control label to an aggregate of control inputs. The literature on symbolic models for control systems is very broad. Early results were based on dynamical consistency properties \cite{caines}, natural invariants of the control system \cite{koutsoukos}, $l$-complete approximations \cite{moor}, and quantized inputs and states \cite{forstner,BMP02}. Recent results include work on controllable discrete-time linear systems \cite{TabuadaLTL}, piecewise-affine and multi-affine systems \cite{habets,BH06}, set-oriented discretization approach for discrete-time nonlinear optimal control problems \cite{junge1}, abstractions based on convexity of reachable sets \cite{gunther}, incrementally stable and incrementally forward complete nonlinear control systems with and without disturbances \cite{PolaAutom2008,MajidTAC11,PolaSIAM2009,BorriIJC2012}, switched systems \cite{GirardTAC2010} and time-delay systems \cite{PolaSCL10,PolaIJRNC2014}. 
A limitation of some of the above results is that in practice they can only be applied to control systems with small dimensional state space. This is because the computational complexity arising in the construction of symbolic models often scales exponentially with the dimension of the state space of the control system considered. When internal interconnection structure of a control system is known, one can make use of this information with the purpose of reducing the computational complexity in deriving symbolic models. Indeed, once a symbolic model is constructed for each subsystem, one can then simply interconnect them to obtain a symbolic model of the original control system. 
In this paper we follow this approach and propose a network of symbolic models that approximates a network of discrete--time nonlinear control systems. 
In particular, if each subsystem composing the network admits an incremental input--to--state stable Lyapunov function and if some small gain theorem--type conditions are satisfied, a network of symbolic models, each one associated with each subsystem composing the network, is proposed which is approximately bisimilar to the original network with any desired accuracy. Quantization parameters of the symbolic models are derived on the basis of the topological properties of the network. Advantages of the proposed approach with respect to current literature are as follows. Firstly, our approach does not cancel topological properties of the network, which can be of great importance in the design process; for example, it allows incremental re-design of the system when new functionalities, e.g. energy sustainability or security, are added to an existing design or an error is discovered late in the design process. 
Secondly, the proposed approach simplifies the construction of symbolic models. Indeed we only require the knowledge of a $\delta$--ISS Lyapunov function $V_i$ for each subsystem $\Sigma_i$, and the satisfaction of some small gain theorem--type conditions for the strongly connected aggregates of subsystems. A single $\delta$--ISS Lyapunov function for the entire network is not needed to be found. This is especially useful when real-word complex systems are considered. From the computational complexity point of view, since we do not construct a symbolic model of the entire network, but symbolic models of each subsystem, whose composition approximates the original network for any desired accuracy, the resulting computational complexity scales linearly with the number of subsystems composing the network. We stress that composing symbolic models in the network is not always necessary for control design (and formal verification) purposes. In fact, by using the so--called \textit{on--the--fly} algorithms (e.g. \cite{onthefly3,onthefly2}, see also \cite{PolaTAC12}), a symbolic controller for the whole network can be designed without the need of constructing explicitly the whole symbolic model of the network. \\
Symbolic models for interconnected systems have been also proposed in \cite{Tazaki08}. This paper compares as follows with \cite{Tazaki08}. While \cite{Tazaki08} considers stabilizable input--state--output linear systems, this paper considers $\delta$--ISS nonlinear control systems. Moreover, while in \cite{Tazaki08} dynamical properties of control systems are not found for the quantization parameters to match certain conditions guaranteeing existence of approximately bisimilar symbolic models, this paper overcomes this drawback and identifies in small gain theorem--type conditions the key ingredient to construct approximately bisimilar networks of symbolic models. 

\section{Networks of Control Systems}

In this paper we consider a network of control systems given by the coupled difference equations $\Sigma_1,\Sigma_2,...,\Sigma_N$ described by:
\begin{equation}
\label{NetCS}
\Sigma_i :
\left\{
\begin{array}
{l}
x_{i}(t+1)=f_i(x_1(t),x_2(t),...,x_N(t),u_i(t)),\\
x_{i}(t) \in \mathcal{X}_i \subset \mathbb{R}^{n_i},u_{i}(t) \in \mathcal{U}_i \subset \mathbb{R}^{m_i}, t \in \mathbb{N}_{0}. 
\end{array}
\right.
\end{equation}

\noindent
Let $n=\sum_{i\in [1;N]}n_i$ and $m=\sum_{i\in [1;N]}m_i$. 
Functions $f_{i}:\mathbb{R}^{n} \times \mathbb{R}^{m_i} \rightarrow \mathbb{R}^{n_i}$ are assumed to be locally Lipschitz and satisfying $f_i(0_{n},0_{m_i})=0_{n_i}$.
Sets $\mathcal{X}_{i}$ and $\mathcal{U}_{i}$ are assumed to be convex, bounded and with interior.
For compact notation we refer to the network of control systems in (\ref{NetCS}) by the control system $\Sigma$ described by $x(t+1) = f(x(t),u(t))$, $x(t)\in \mathcal{X}\subset \mathbb{R}^{n}$, $u(t)\in \mathcal{U}\subset \mathbb{R}^{m}$, $t\in\mathbb{N}_{0}$, where $\mathcal{X}:=\times_{i\in [1;N]} \mathcal{X}_{i}$, $\mathcal{U}:=\times_{i\in [1;N]} \mathcal{U}_{i}$ and $f(x,(u_1,u_2,...,u_N)):=(f_1(x,u_1),f_2(x,u_2),...,f_N(x,u_N))$ for any $x\in \mathbb{R}^{n}$ and $(u_1,u_2,...,$ $u_N)\in \mathbb{R}^{m}$. Notation and some technical notions used in the sequel are reported in the Appendix.

\section{Results}

Define the directed graph $\mathcal{G}=(\mathcal{V},\mathcal{E})$ where $\mathcal{V}=[1;N]$ and $(i,j)\in \mathcal{E}$, if function $f_{j}$ of $\Sigma_{j}$ depends explicitly on variable $x_{i}$ or equivalently, there exist $y_{i},z_{i}\in \mathcal{X}_{i}$ such that  $f_{j}(x_1,...,x_{i-1},y_{i},$ $x_{i+1},...,x_n,u_{j}) \neq f_{j}(x_1,...,x_{i-1},z_{i},x_{i+1},...,x_n,u_{j})$.  
Let $\SCC(\mathcal{G})$ be the collection of strongly connected components $\Scc_k$ associated with $\mathcal{G}$; we define $\Scc_k=(\mathcal{V}_k,\mathcal{E}_k)$, $\overline{N}_k=\card(\mathcal{V}_k)$, $\overline{N}=\card(\SCC(\mathcal{G}))$ and  $\mathcal{V}_k=\{i(1,k),i(2,k),...,i(\overline{N}_k,k)\}$. We recall that by contracting each $\Scc_k$ to a vertex, a Directed Acyclic Graph ($\DAG$) is obtained. Given $\Scc_k \in \Scc \subseteq \SCC(\mathcal{G})$ we denote by $\Post(\Scc_k)$ the collection of strongly connected components that can be reached in one step by $\Scc_k$ and by $\Leaves(\Scc)$ the collection of $\Scc_k\in \Scc$ for which $\Post(\Scc_k)=\varnothing$. We denote by $\Post^{-1}$ the inverse map of operator $\Post$, i.e. $\Scc_k \in \Post^{-1}(\Scc)$ if and only if $\Scc \subseteq \Post (\Scc_k)$. For each $\Scc_k \in \SCC(\mathcal{G})$, define $\Xi_k= \times_{i\in\mathcal{V}_k} \mathcal{X}_i$, $\Omega_k= \times_{i\in\mathcal{V}_k} \mathcal{U}_{i}$, $\overline{n}_k=\sum_{i\in \mathcal{V}_k} n_{i}$ and $\overline{m}_k=\sum_{i\in\mathcal{V}_k}m_i$. Note that sets $\Xi_k$ and $\Omega_k$ are convex, bounded and with interior. 
The interconnection of control systems $\Sigma_{i(1,k)},\Sigma_{i(2,k)},...,\Sigma_{i(\overline{N}_k,k)}$ associated with each $\Scc_k\in \SCC(\mathcal{G})$, is denoted by 
\begin{equation}
\label{SCCi}
\Sigma_{\Scc_k} :
\left\{
\begin{array}
{l}
\xi_k(t+1)=\varphi_k(\xi_1(t),\xi_2(t),...,\xi_{\overline{N}}(t),\omega_k(t)),\\
\xi_k(t) \in \Xi_k\subset \mathbb{R}^{\overline{n}_{k}}, \omega_k(t) \in \Omega_k \subset \mathbb{R}^{\overline{m}_{k}}, t\in\mathbb{N}_{0}, 
\end{array}
\right.
\end{equation}

\noindent
where $\varphi_k: \mathbb{R}^n \times \mathbb{R}^{\overline{m}_{k}} \rightarrow  \mathbb{R}^{\overline{n}_k}$. 
The compositional approach that we take to build a network of symbolic models for $\Sigma$ in (\ref{NetCS}) is based on the following three steps: 
\textit{(Step \#1)} Construction of symbolic models for $\Sigma_i$ in Section \ref{subsec1}; \textit{(Step \#2)} Construction of symbolic models for $\Sigma_{\Scc_k}$ in Section \ref{subsec2}; \textit{(Step \#3)} Construction of symbolic models for $\Sigma$ in Section \ref{subsec3}.

\subsection{Symbolic models for subsystems $\Sigma_{i}$} \label{subsec1}
We start by providing a representation of each subsystem $\Sigma_i$ ($i\in [1;N]$) in terms of the system\footnote{The notion of system, taken from \cite{paulo}, is reported in the Appendix.} 
$S(\Sigma_i)=(X^{*}_i,W^{*}_i \times U^{*}_i,\rTo_{*,i},Y^{*}_i,H^{*}_i)$ where $X^{*}_i=\mathcal{X}_i$, 
$W^{*}_i=\mathcal{X}_1 \times \mathcal{X}_2 \times ... \times \mathcal{X}_{i-1} \times \mathcal{X}_{i+1} \times ... \times \mathcal{X}_{N}$, $U^{*}_i=\mathcal{U}_i$, $x_i \rTo_{*,i}^{(x_1,...,x_{i-1},x_{i+1},...,x_N,u_i)} x^{+}_i$ if $x^{+}_i=f_i(x_1,x_2,...,x_N,u_i)$, $Y^{*}_i=\mathcal{X}_i$ and $H^{*}_i(x_i)=x_i$. System $S(\Sigma_i)$ preserves many important properties of control system $\Sigma_i$, as for example reachability properties. System $S(\Sigma_i)$ is metric when we regard $Y^{*}_i=\mathcal{X}_i$ as being equipped with the metric $\mathbf{d}_i(x_i,x'_i)=\Vert x_i-x'_i\Vert$. Note that system $S(\Sigma_i)$ is not symbolic because the cardinality of sets $X^{*}_i$, $W^{*}_i$ and $U^{*}_i$ is infinite. We now define a suitable symbolic system that will approximate $S(\Sigma_i)$ with any desired precision. 

\begin{definition} 
Given $\Sigma_i$, $i\in [1;N]$ and a quantization vector $\eta \in\mathbb{R}^{+}_N$, define the system $S^{\eta}(\Sigma_i)=(X^{\eta}_i,W^{\eta}_i\times U^{\eta}_i,\rTo_{\eta,i},Y^{\eta}_i,H^{\eta}_i)$ where $X^{\eta}_i=[\mathcal{X}_i]_{\eta(i)}$, $W^{\eta}_i=[\mathcal{X}_1]_{\eta(1)} \times [\mathcal{X}_2]_{\eta(2)} \times ... \times [\mathcal{X}_{i-1}]_{\eta(i-1)} \times [\mathcal{X}_{i+1}]_{\eta(i+1)} \times ... \times [\mathcal{X}_N]_{\eta(N)}$, $U^{\eta}_i=[\mathcal{U}_i]_{\eta(i)}$, $x_i \rTo^{x_1,x_2,...,x_{i-1},x_{i+1},..., x_N,u_i}_{\eta,i} x^{+}_i$ if $x^{+}_i = [f_i(x_1,x_2,..., x_N,u)]_{\eta(i)}$, $Y^{\eta}_i=\mathcal{X}_i$ and $H^{\eta}_i(x_i)=x_i$.
\end{definition}

System $S^{\eta}(\Sigma_i)$ is metric when we regard $Y^{\eta}_i=\mathcal{X}_i$ as being equipped with the metric $\mathbf{d}_i(x_i,x'_i)=\Vert x_i-x'_i\Vert$. Moreover, since sets $\mathcal{X}_i$ and $\mathcal{U}_i$ are bounded then sets $X^{\eta}_i$, $W^{\eta}_i$ and $U^{\eta}_i$ are finite from which, system $S^{\eta}(\Sigma_i)$ is symbolic. Space and time complexity in computing the symbolic model $S^{\eta}(\Sigma_i)$ are given by 
$\scomplex (S^{\eta}(\Sigma_i))=\card (X^{\eta}_i)^{2} \cdot \card (W^{\eta}_i\times U^{\eta})$ and $\tcomplex (S^{\eta}(\Sigma_i))=\card (X^{\eta}_i) \cdot \card (W^{\eta}_i\times U^{\eta})$, respectively. 
In the sequel, we consider the following assumption:\\
(A1) For each $i\in [1;N]$, a locally Lipschitz function $V_i: \mathbb{R}^{n_i} \times \mathbb{R}^{n_i} \rightarrow \mathbb{R}^{+}_{0}$ exists for control system $\Sigma_i$, which satisfies the following inequalities for some $\mathcal{K}_{\infty}$ functions $\underline{\alpha}_i$, $\overline{\alpha}_i$, $\rho_i$ and $\mathcal{K}$ functions $\sigma_i$ and $\sigma_{i,j}$ ($j\in [1;N], i\neq j$):
\begin{itemize}
\item[(i)] $\underline{\alpha}_i(\left\Vert x_i-x'_i \right\Vert )\leq V_i(x_i,x'_i)\leq \overline{\alpha}_i(\left\Vert x_i-x'_i \right\Vert )$, for any $x_i,x_i'\in\mathbb{R}^{n_i}$; 
\item[(ii)] $V_i(f_i(x_1,x_2,...,x_N,u_{i}),f_i(x'_1,x'_2,...,x'_N,u'_{i})) - V_i(x_i,x'_i) \leq -\rho_i ( V_i(x_i,x'_i) ) + \\
\sum_{j\in [1;N],j\neq i} \sigma_{i,j} (\Vert x_j - x'_j \Vert ) + \sigma_i (\left\Vert u_{i}-u_{i}' \right\Vert )$, for any $x_j,x'_j\in\mathbb{R}^{n_j}$ ($j\in [1;N]$) and any $u_{i},u_{i}'\in \mathbb{R}^{m_{i}}$.
\end{itemize}
Function $V_i$ is called a $\delta$--ISS Lyapunov function \cite{IncrementalS,BayerECC2013} for control system $\Sigma_i$. The above assumption has been shown in \cite{BayerECC2013} to be a sufficient condition for the control system $\Sigma_i$ to fulfill the incremental input--to--state stability property \cite{IncrementalS,BayerECC2013}. We can now give the following preliminary result.

\begin{proposition}
\label{Th1}
Suppose that Assumption (A1) holds and let $L_i$ be a Lipschitz constant of function $V_i$ in $\mathcal{X}_i \times \mathcal{X}_i$. Then, for any desired precision $\varepsilon_i\in\mathbb{R}^{+}$ and for any $\eta \in \mathbb{R}^{+}_{N}$ satisfying the following inequalities
\begin{eqnarray}
& & L_i\, \eta(i) + \sum_{j\in [1;N],j\neq i} \sigma_{i,j} ( \eta(j) ) + \sigma_i(\eta(i)) \leq (\rho_i \circ \underline{\alpha}_i)(\varepsilon_i),\label{statem} \\
& & \overline{\alpha}_i(\eta(i))\leq \underline{\alpha}_i(\varepsilon_i), \label{statem1}
\end{eqnarray}
systems $S(\Sigma_i)$ and $S^{\eta}(\Sigma_i)$ are approximately bisimilar\footnote{The notion of approximate bisimulation, taken from \cite{AB-TAC07}, is recalled in the Appendix.} with precision $\varepsilon_i$.
\end{proposition}

The proof can be given along the lines of the proof of Theorem 5.1 in \cite{PolaAutom2008}. 
We include it here for the sake of completeness. 

\begin{proof} 
Consider the relation $\mathcal{R}_i\subseteq X^{*}_i\times X^{\eta}_i$ defined by $(x_i,x'_i)\in\mathcal{R}_i$ if and only if $V_i(x_i,x'_i) \leq \underline{\alpha}_i(\varepsilon_i)$ and consider any pair $(x_i,x'_i)\in\mathcal{R}$. We first note that $\Vert x_i-x'_i \Vert \leq \underline{\alpha}_i^{-1}(V_i(x_i,x'_i)) \leq \varepsilon_i$ from which, condition (i) of Definition \ref{ASR} holds. We now show that also condition (ii) holds. Consider any $(x_1,...,x_{i-1},x_{i+1},...,x_N,u_i)\in W^{*}_i \times U^{*}_i$ and the transition $x_i \rTo^{x_1,...,x_{i-1},x_{i+1},...,x_N,u_i}_{*,i} x^{+}_i$ in system $S(\Sigma_i)$. Consider a control label $(x'_1,...,x'_{i-1},x'_{i+1},$ $...,x'_N,u'_i)\in W^{\eta}_i \times U^{\eta}_i$ such that $\Vert x_j-x'_j\Vert \leq\eta(j)$ for any $j\in [1;N]$, $j \neq i$ and $\Vert u_i-u'_i\Vert \leq\eta(i)$. 
Set $z_i=f_i(x'_1,x'_2,...,x'_N,u'_i)$ and $x'^{+}_i=[z_i]_{\eta(i)}$, and consider the transition $x'_i \rTo_{\eta,i}^{x'_1,...,x'_{i-1},x'_{i+1},...,x'_N,u'_i} x'^{+}_i$ in system $S^{\eta}(\Sigma_i)$. We get 
$
V_i(x^{+}_i,x'^{+}_i) \leq V_i(x^{+}_i,z_i) + L_i\, \Vert x'^{+}_i -z_i \Vert \leq 
V_i(x_i,x'_i) - \rho_i(V_i(x_i,x'_i)) + \sum_{j\in [1;N],j\neq i} \sigma_{i,j} (\Vert x_j - x'_j \Vert ) + \sigma_i(\Vert u_i - u'_i \Vert) 
+ L_i\, \eta(i)\leq
(\Id-\rho_i) \circ \underline{\alpha}_i(\varepsilon_i) + \sum_{j\in [1;N], j\neq i} \sigma_{i,j}(\eta(j)) + \sigma_i(\eta(i))+ L_i\, \eta(i)
\leq \underline{\alpha}_i(\varepsilon_i)
$. 
In particular, the first inequality holds by definition of $L_i$, the second inequality by the inequality (ii) in Assumption (A1), the third inequality by the definition of $x'^{+}_i$ and the last inequality by condition (\ref{statem}). Hence, condition (ii) in Definition \ref{ASR} holds. Condition (iii) in Definition \ref{ASR} can be shown by using similar arguments. Finally, for any $x_i\in X^{*}_i$ by choosing $x'_i=[x_i]_{\eta(i)}\in X^{\eta}_i$ we get $V_i(x_i,x'_i) \leq \overline{\alpha}_i(\Vert x_i - x'_i\Vert) \leq \overline{\alpha}_i(\eta(i)) \leq \underline{\alpha}_i(\varepsilon_i)$. In particular, the first inequality in the above chain holds by the inequality (i) in the statement and the last one by condition (\ref{statem1}). Hence, $\mathcal{R}_i(X^{*}_i)=X^{\eta}_i$. Conversely, for any $x'_i \in X^{\eta}_i$ by picking $x_i=x'_i$ one gets $V_i(x_i,x'_i)=0\leq \underline{\alpha}_i(\varepsilon_i)$ from which, $\mathcal{R}_i^{-1}(X^{\eta}_i)=X^{*}_i$, which concludes the proof.  
\end{proof}

\subsection{Symbolic models for interconnected subsystems $\Sigma_{\Scc_k}$} \label{subsec2}
As in the previous section, we start by providing a representation of each subsystem $\Sigma_{\Scc_k}$ in terms of the system $S(\Sigma_{\Scc_k})=(X^{*}_{\Scc_k},W^{*}_{\Scc_k} \times U^{*}_{\Scc_k},\rTo_{*,\Scc_k},Y^{*}_{\Scc_k},H^{*}_{\Scc_k})$ where $X^{*}_{\Scc_k}=\Xi_k$, 
$W^{*}_{\Scc_k}=\Xi_1 \times \Xi_2 \times ... \times \Xi_{k-1} \times \Xi_{k+1} \times ... \times \Xi_{\overline{N}}$, $U^{*}_{\Scc_k}=\Omega_k$, $\xi_k \rTo_{*,\Scc_k}^{(\xi_1,...,\xi_{k-1},\xi_{k+1},...,\xi_{\overline{N}},\omega_k)} \xi^{+}_k$ if $\xi^{+}_k=\varphi_k(\xi_1,\xi_2,...,\xi_{\overline{N}},\omega_k)$, $Y^{*}_{\Scc_k}=\Xi_k$ and $H^{*}_{\Scc_k}(\xi_k)=\xi_k$. 
System $S(\Sigma_{\Scc_k})$ is metric when we regard $Y^{*}_{\Scc_k}=\Xi_k$ as being equipped with the metric $\mathbf{d}(\xi_k,\xi'_k)=\max_{i\in \mathcal{V}_k}\Vert x_i-x'_i\Vert$ for any $\xi_k:=(x_{i(1,k)},x_{i(2,k)},...,x_{i(\overline{N}_k,k)}),\xi'_k:=(x'_{i(1,k)},x'_{i(2,k)},...,x'_{i(\overline{N}_k,k)}) \in \Xi_k$.
In the sequel we consider the following technical assumption that has been used in \cite{SmallGainTh2} to prove the small gain theorem for ISS continuous--time control systems:\\
(A2) 
There exist $\mathcal{K}_{\infty}$ functions $g_i^k$, reals $a_i^k\in \mathbb{R}^+$ and $c^k_{ij} \in \mathbb{R}_0^+$, $i,j\in \mathcal{V}_k$, $j\ne i$, such that 
$\rho_{i}(s)\geq a_{i}^k g_{i}^k(s)$ and $\sigma_{i,j} \circ \underline{\alpha}^{-1}_j (s)\leq c_{ij}^k g_{j}^k(s)$, for any $i,j\in \mathcal{V}_k$, $j\ne i$. \\
The above assumption is standard in the literature concerning the stability of network of control systems studied by means of small gain arguments (see, for instance, \cite{SmallGainTh2} for the case of ordinary differential equations). In our discrete--time case, such assumption holds, for instance, if functions $f_i$ with $i\in \mathcal{V}_k$ are globally Lipschitz, and Assumption (A1) holds with $V_i(x_i,x_i')=\Vert x_i-x_i'\Vert$ for any $i\in \mathcal{V}_k$. 
This reasoning is applied in Section \ref{Sec:example} to an academic example.\\
For later use, define $V^{\vecc}_{k}(\xi_k,\xi'_k)=(V_{i(1,k)}(x_{i(1,k)},x'_{i(1,k)}),V_{i(2,k)}(x_{i(2,k)},x'_{i(2,k)}),$ $...,V_{i(\overline{N}_k,k)}(x_{i(\overline{N}_k,k)},$ $x'_{i(\overline{N}_k,k)}))$ where $\xi_k=(x_{i(1,k)},x_{i(2,k)},...,$ $x_{i(\overline{N}_k,k)})$ and $\xi'_k=(x'_{i(1,k)},x'_{i(2,k)},...,$ $x'_{i(\overline{N}_k,k)})$, $A_{k}=\diag(a_{i(1,k)}^k,a_{i(2,k)}^k,...,a_{i(\overline{N}_k,k)}^k)$, and $g^k(s)=(g_{i(1,k)}^k(s_1),g_{i(2,k)}^k(s_2),...,g_{i(\overline{N}_k,k)}^k (s_{\overline{N}_k}))$, for any $s=(s_{1},s_{2},...,s_{\overline{N}_k})\in \mathbb{R}^{+}_{\overline{n}_k}$. Moreover, define matrix $C_k$ such that entries in the diagonal are $0$ and the entry of row $j$ and column $j'$ with $j\neq j'$ is given by $c^k_{i(j,k)i(j',k)}$, for all $j,j'\in [1;\overline{N}_k]$. We can now give the following result. 

\begin{theorem}
\label{main2}
Let us consider the subsystem $\Sigma_{\Scc_k}$.
If Assumptions (A1) and (A2) and the inequality $r(A_k^{-1}C_k)<1$ hold, 
then, for any vector $\underline \lambda_k=(\lambda_{i(1,k)},\lambda_{i(2,k)},...,\lambda_{i(\overline{N_k},k)}) \in \mathbb{R}^{+}_{\overline{N}_k}$ satisfying $\underline \lambda_k^T(A_k-C_k)>0$, function $\overline{V}_k(\xi_k,\xi'_k)=\underline{\lambda}_k^{T}V^{\vecc}_k(\xi_k,\xi'_k)$, $\xi_k,\xi'_k\in\mathbb{R}^{\overline{n}_k}$ is a $\delta$--ISS Lyapunov function for $\Sigma_{\Scc_k}$, i.e. it satisfies the following inequalities, 

\begin{itemize}
\item[(i)] $\underline{\alpha}^k(\left\Vert \xi_k-\xi'_k \right\Vert )\leq \overline{V}_k(\xi_k,\xi'_k)\leq \overline{\alpha}^k(\left\Vert \xi_k - \xi'_k \right\Vert )$, for any $\xi_k,\xi'_k\in\mathbb{R}^{\overline{n}_k}$; 
\item[(ii)] $\overline{V}_k(\varphi_k(\xi_1,\xi_2,...,\xi_{\overline{N}},\omega_{k}),\varphi_k(\xi'_1,\xi'_2,...,\xi'_{\overline{N}},\omega'_{k})) - \overline{V}_k(\xi_k,\xi'_k) \leq -\rho^k ( \overline{V}_k(\xi_k,\xi'_k) ) + \\
\sum_{j\in [1;\overline{N}],j \neq k} \sigma_{j}^k (\Vert \xi_j - \xi'_j \Vert) + \sigma^k (\left\Vert \omega_{k} - \omega_{k}' \right\Vert )$, for any $\xi_k,\xi'_k\in\mathbb{R}^{\overline{n}_k}$ and any $\omega_{k},\omega_{k}'\in \mathbb{R}^{\overline{m}_{k}}$,
\end{itemize}
for some $\mathcal{K}_{\infty}$ functions $\underline{\alpha}^k$, $\overline{\alpha}^k$, $\rho^k$ and $\mathcal{K}$ functions $\sigma^{k}$, $\sigma_{j}^k$ ($j\in [1;\overline{N}], j\neq k$). 
Moreover, let $L^k$ be a Lipschitz constant of function $\overline{V}_k$ in $\Xi_k \times \Xi_k$.
For any desired precision $\varepsilon^k\in\mathbb{R}^{+}$, select vector $\overline{\eta} \in\mathbb{R}^{+}_{\overline{N}}$ satisfying the 
following inequalities:
\begin{eqnarray}
& & L^k\, \overline{\eta}(k) + \sum_{j\in [1;\overline{N}], j \neq k} \sigma_{j}^{k} ( \overline{\eta}(j) ) + \sigma^k(\overline{\eta}(k)) \leq (\rho^k \circ \underline{\alpha}^k)(\varepsilon^k),\label{Sstatem}\\
& & \overline{\alpha}^k(\overline{\eta}(k))\leq \underline{\alpha}^k(\varepsilon^k). \label{Sstatem1}
\end{eqnarray}

\noindent
Define vector $\eta\in\mathbb{R}^{+}_N$ by $\eta(i)=\overline{\eta}(k)$ for all $i \in \mathcal{V}_k$ and $k\in [1;\overline{N}]$. Then, the composition\footnote{The definition of the composition operator $\mathcal{S}(.)$ is reported in the Appendix.} $\mathcal{S}(\{S^{\eta}(\Sigma_i)\}_{i\in \mathcal{V}_k})$ of the symbolic models $S^{\eta}(\Sigma_i)$ associated with each control system $\Sigma_i$ ($i\in \mathcal{V}_k$), is approximately bisimilar with precision $\varepsilon^k$ to system $S(\Sigma_{\Scc_k})$.
\end{theorem}
\begin{proof}
The first part of the proof follows the proof of Theorem 4.7 in \cite{SmallGainTh2}. 
By Lemma 3.1 in \cite{SmallGainTh2} if $r(A_k^{-1}C_k)<1$ there exists a vector $\underline{\lambda}_k=(\lambda_{i(1,k)},\lambda_{i(2,k)},...,$ $\lambda_{i(\overline{N_k},k)}) \in \mathbb{R}^{+}_{\overline{N}_k}$ such that $\underline{\lambda}_k^{T}(-A_k+C_k)<0$. 
By defining $\underline{\alpha}^{k}(s)=\min_{\Vert(s_{i(1,k)},s_{i(2,k)},...,s_{i(\overline{N}_k,k)})\Vert=s}$ $\sum_{i\in \mathcal{V}_k} \lambda_i \underline{\alpha}_{i}(s_{i})$ and $\overline{\alpha}^{k}(s)=\max_{\Vert(s_{i(1,k)},s_{i(2,k)},...,s_{i(\overline{N}_k,k)})\Vert=s}$ $\sum_{i\in \mathcal{V}_k} \lambda_i \overline{\alpha}_{i}(s_{i})$  ($s_i\in\mathbb{R}^{+}_{0}$), the inequality (i) in the statement holds. We now show inequality (ii). 
Consider any 
$\xi_j:=(x_{i(1,j)},x_{i(2,j)},...,$ $x_{i(\overline{N}_j,j)}),\xi'_j:=(x'_{i(1,j)},x'_{i(2,j)},...,x'_{i(\overline{N}_j,j)})\in \mathbb{R}^{\overline{n}_j}$ and
$\omega_j:=(u_{i(1,j)},u_{i(2,j)},...,u_{i(\overline{N}_j,j)})$, $\omega'_j:=(u'_{i(1,j)},u'_{i(2,j)},...,$ $u'_{i(\overline{N}_j,j)})
\in \mathbb{R}^{\overline{m}_j}$. 
Under Assumptions (A1) and (A2) the following equalities/inequalities hold:\\
$
\overline{V}_k(\varphi_k(\xi_1,\xi_2,...,\xi_{\overline{N}},\omega_{k}),\varphi_k(\xi'_1,\xi'_2,...,\xi'_{\overline{N}},\omega'_{k})) - \overline{V}_k(\xi_k,\xi'_k)= \\
\sum_{i\in \mathcal{V}_k}\lambda_i ( V_i(f_i(x_1,x_2,...,x_N,u_i),f_i(
x'_1,x'_2,...,x'_N,u'_i))- V_i(x_i,x'_i))  \leq \\
\sum_{i\in \mathcal{V}_k}\lambda_i ( 
-\rho_i ( V_i(x_i,x'_i) ) + \sum_{j\in [1;N],j\neq i} \sigma_{i,j} (\Vert x_j - x'_j \Vert )) + \sigma_i (\left\Vert u_{i}-u_{i}' \right\Vert )
) = \\ 
\sum_{i\in \mathcal{V}_k}\lambda_i ( 
-\rho_i ( V_i(x_i,x'_i) ) + \sum_{j\in \mathcal{V}_k,j\neq i} \sigma_{i,j} (\Vert x_j - x'_j \Vert ) + \\
\quad \sum_{j\in [1;N] \backslash \mathcal{V}_k} \sigma_{i,j} (\Vert x_j - x'_j \Vert ))
 + \sum_{i\in \mathcal{V}_k} \lambda_i \sigma_i (\left\Vert u_{i}-u_{i}' \right\Vert )
) \leq \\
\sum_{i\in \mathcal{V}_k}\lambda_i ( 
-\rho_i ( V_i(x_i,x'_i) ) + \sum_{j\in \mathcal{V}_k,j\neq i} \sigma_{i,j} \circ \underline{\alpha}^{-1}_j (V_j(x_j,x'_j)) + \\
\quad \sum_{j\in [1;N] \backslash \mathcal{V}_k} \sigma_{i,j} (\Vert x_j - x'_j \Vert ))
 + \sum_{i\in \mathcal{V}_k} \lambda_i \sigma_i (\left\Vert u_{i}-u_{i}' \right\Vert )
) \leq \\
\sum_{i\in \mathcal{V}_k} \lambda_i ( -a_i^k g_i^k( V_i(x_i,x'_i) ) + \sum_{j\in \mathcal{V}_k,j\neq i} c_{ij}^k g_j^k( V_j(x_j,x'_j) )) + \\
\sum_{i\in \mathcal{V}_k} \lambda_i (\sum_{j\in [1;N] \backslash \mathcal{V}_k} \sigma_{i,j} (\Vert x_j - x'_j \Vert ) )+
\sum_{i\in \mathcal{V}_k} \lambda_i\sigma_i (\Vert u_{i}-u_{i}' \Vert ) ) =\\
\underline{\lambda}_k^{T} (-A_k+C_k) g^k(V^{\vecc}_k(\xi_k,\xi'_k))+
\sum_{i\in \mathcal{V}_k}
\sum_{j\in [1;\overline{N}], j\neq k} 
\sum_{j'\in \mathcal{V}_j}
\lambda_i \sigma_{i,j'} (\Vert x_{j'} - x'_{j'} \Vert ) +
\sum_{i\in \mathcal{V}_k}\lambda_i \sigma_i (\left\Vert u_{i}-u_{i}' \right\Vert ) =
\underline{\lambda}_k^{T} (-A_k+C_k) g^k(V^{\vecc}_k(\xi_k,\xi'_k))+
\sum_{j\in [1;\overline{N}], j\neq k} 
\left(
\sum_{i\in \mathcal{V}_k}
\sum_{j'\in \mathcal{V}_j}
\lambda_i \sigma_{i,j'} (\Vert x_{j'} - x'_{j'} \Vert ) 
\right)+
\sum_{i\in \mathcal{V}_k}\lambda_i \sigma_i (\left\Vert u_{i}-u_{i}' \right\Vert ) .\\
$
By defining 
$\rho^k (s)=\min\{-\underline{\lambda}_k^{T}(-A_k+C_k)g^k(V^{\vecc}_k (\xi_k,\xi'_k))|$ $\underline{\lambda}_k^{T}V^{\vecc}_k(\xi_k,\xi'_k)=s\}$ ($s\in\mathbb{R}^{+}_{0}$), 
$\sigma^k_{j}(s)=\max_{\Vert(s_1,s_2,...,s_{\overline{N}_j})\Vert=s}$ 
$\sum_{i\in \mathcal{V}_k}$ 
$\sum_{j'\in \mathcal{V}_j , j\neq k} \lambda_i \sigma_{i,j'} (s_{j'})$ ($s_{j'}\in\mathbb{R}^{+}_{0}$),  
$\sigma^k(s)=$ 
$\max_{\Vert(s_1,s_2,...,s_{\overline{N}_k})\Vert=s}$ $\sum_{i\in \mathcal{V}_k} \lambda_i \sigma_{i}(s_i)$ ($s_i\in\mathbb{R}^{+}_{0}$), 
one gets 
$\overline{V}_k(\varphi_k(\xi_1,\xi_2,...,\xi_{\overline{N}},\omega_{k}),\varphi_k(\xi'_1,\xi'_2,$ $...,\xi'_{\overline{N}},\omega'_{k})) - \overline{V}_k(\xi_k,\xi'_k) \leq -\rho^k ( \overline{V}_k(\xi_k,\xi'_k) ) + 
\sum_{j\in [1;\overline{N}],j\neq k} \sigma_{j}^k (\Vert \xi_j - \xi'_j \Vert) + \sigma^k (\left\Vert \omega_{k} - \omega_{k}' \right\Vert )$.
Since $\sigma^k$ and $\sigma^k_{j}$ are $\mathcal{K}$ and $\rho_k$ is $\mathcal{K}_{\infty}$, the inequality (ii) in the statement holds and hence, 
$\overline{V}_k$ is a $\delta$--ISS Lyapunov function for $\Sigma_{\Scc_k}$. We now show the second part of the statement. To this purpose define 
the system $S^{\overline{\eta}}(\Sigma_{\Scc_k})=(X^{\overline{\eta}}_{\Scc_k},W^{\overline{\eta}}_{\Scc_k}\times U^{\overline{\eta}}_{\Scc_k},
\rTo_{\overline{\eta},\Scc_k},X^{\overline{\eta}}_{\Scc_k},H^{\overline{\eta}}_{\Scc_k})$ where $X^{\overline{\eta}}_{\Scc_k}=[\Xi_k]_{\overline{\eta}(k)}$, 
$W^{\overline{\eta}}_{\Scc_k}
=[\Xi_1]_{\overline{\eta}(1)} \times [\Xi_2]_{\overline{\eta}(2)} \times ... \times [\Xi_{k-1}]_{\overline{\eta}(k-1)} \times [\Xi_{k+1}]_{\overline{\eta}(k+1)} \times ... \times [\Xi_{\overline{N}}]_{\overline{\eta}(\overline{N})}$, $U^{\overline{\eta}}_{\Scc_k}=[\Omega_k]_{\overline{\eta}(k)}$, 
$\xi_k \rTo ^{\xi_1,\xi_2,...,\xi_{k-1},\xi_{k+1},..., \xi_{\overline{N}},\omega_k}_{\overline{\eta},\Scc_k} \xi_{k}^{+}$ if 
$\xi^{+}_k = [\varphi_{k} (\xi_1,\xi_2,..., \xi_{\overline{N}},\omega_k)]_{\overline{\eta}(k)}$, $Y^{\overline{\eta}}_{\Scc_k}=X^{\overline{\eta}}_{\Scc_k}$, and $H^{\overline{\eta}}_{\Scc_k}(\xi_k)=\xi_k$. By using the same arguments as in Proposition \ref{Th1}, for any $\overline{\eta} \in\mathbb{R}^{+}_{\overline{N}}$ satisfying the inequalities in (\ref{Sstatem}) and (\ref{Sstatem1}), we get $S(\Sigma_{\Scc_k}) \cong_{\varepsilon^k} S^{\overline{\eta}}(\Sigma_{\Scc_k})$. Finally, since $S^{\overline{\eta}}(\Sigma_{\Scc_k})= \mathcal{S}(\{S^{\eta}(\Sigma_i)\}_{i\in \mathcal{V}_k})$, the second part of the statement is proven. 
\end{proof}

\incmargin{1em}
\restylealgo{boxed}\linesnumbered
\begin{algorithm*}
\label{alg}
\SetLine
\caption{Compositional design of quantization parameters.}
\textbf{select the desired precision} $\varepsilon\in\mathbb{R}^{+}$\;
\textbf{set} $\overline{\eta}(k):=\infty$, $\eta^{\ast}_k:=\infty$, $\varepsilon^k:=0$, $\forall k\in [1;\overline{N}]$; $\SCC_{\temp} := \SCC(\mathcal{G})$\;
\While{$\SCC_{\temp} \neq \varnothing$}
{
\ForEach {$\Scc_k \in \Leaves(\SCC_{\temp})$}
{
\If{$\overline{\eta}(k)=\infty$}
{
\eIf{$\SCC_{\temp} = \SCC(\mathcal{G})$}
{
$\varepsilon^k:=\varepsilon$\;
}
{$\varepsilon^k:=\min \{ \overline{\eta}(j), \Scc_j \in \Post(\Scc_k) \}$\;
}
\textbf{select} $\overline{\eta}(k)\in\mathbb{R}^{+}$ and $\eta^{\ast}_j \in\mathbb{R}^{+} ,\forall \Scc_j \in \Post^{-1}(\{\Scc_k\})$ such that:\\
$L^k\,\overline{\eta}(k) + \sigma^k(\overline{\eta}(k)) + \sum_{\Scc_j \in 
\Post^{-1}(\{\Scc_k\})}\sigma^k_j(\eta^{\ast}_j) \leq (\rho^k \circ \underline{\alpha}^k)(\varepsilon^k)$; \\
$\overline{\alpha}^k(\overline{\eta}(k))\leq \underline{\alpha}^k(\varepsilon^k)$\\
\textbf{set} $\overline{\eta}(j) :=\min\{\eta^{\ast}_j,\overline{\eta}(j)\}, \forall \Scc_j \in \Post^{-1}(\{\Scc_k\})$\;
}
}
$\SCC_{\temp}:=\SCC_{\temp} \backslash \Leaves(\SCC_{\temp})$\;
}
\end{algorithm*}
\decmargin{1em}

\subsection{Symbolic models for the network of control systems $\Sigma$} \label{subsec3}

When more than one strongly connected component is associated with $\Sigma$, the following results can be applied.
As in the previous section, we first provide a representation of $\Sigma$ in terms of the system $S(\Sigma)=(X^{*},U^{*},\rTo_{*},Y^{*},H^{*})$ where $X^{*}=\mathcal{X}$, $U^{*}=\mathcal{U}$, $x \rTo_{*}^{u} x^+$ if $x^{+}=f(x,u)$, $Y^{*}=\mathcal{X}$ and $H^{*}(x)=x$. 
System $S(\Sigma)$ is metric when we regard $Y^{*}=\mathcal{X}$ as being equipped with the metric $\mathbf{d}(x,x')=\max_{i\in [1;N]} \Vert x_i-x'_i\Vert$ for any $x:=(x_1,x_2,...,x_N),x':=(x'_1,x'_2,...,x'_N) \in \mathcal{X}$.
Quantization parameters for the network of symbolic models are computed in Algorithm \ref{alg} that is explained in \textit{Step \#3} of the next section, through an academic example. 
It is easy to see that for any chosen precision $\varepsilon\in\mathbb{R}^{+}$, there always exists a vector $\overline{\eta}\in\mathbb{R}^{+}_{\overline{N}}$ of quantization parameters, satisfying conditions in Algorithm \ref{alg}. Moreover, since the number of strongly connected components of $\mathcal{G}$ is finite, Algorithm \ref{alg} terminates in a finite number of steps. We can now give the following result.

\begin{theorem}
\label{ThSCC}
Suppose that Assumption (A1) holds and Assumption (A2) and condition $r(A_k^{-1} C_k)<1$ hold for each $k\in [1;\overline{N}]$. For any desired precision $\varepsilon\in\mathbb{R}^{+}$, let $\overline{\eta}\in\mathbb{R}^{+}_{\overline{N}}$ be obtained as output of Algorithm \ref{alg}. 
Define vector $\eta\in\mathbb{R}^{+}_N$ by $\eta(j)=\overline{\eta}(k)$ for all $j \in \mathcal{V}_k$ and $k\in [1;\overline{N}]$. 
Then, the composition $\mathcal{S}(\{S^{\eta}(\Sigma_i)\}_{i\in [1;N]})$ of the symbolic models $S^{\eta}(\Sigma_i)$ associated with each subsystem $\Sigma_i$ is approximately bisimilar to $S(\Sigma)$ with precision $\varepsilon$.
\end{theorem}

\begin{proof}
Define $S^{\overline{\eta}}(\Sigma_{\Scc_k}):=\mathcal{S}(\{S^{\eta}(\Sigma_j)\}_{j\in \mathcal{V}_k})$ for any $k\in [1;\overline{N}]$. 
First of all note that $\mathcal{S}(\{S^{\eta}(\Sigma_i)\}_{i\in [1;N]})=\mathcal{S}(\{S^{\overline{\eta}}(\Sigma_{\Scc_k})\}_{k\in [1;\overline{N}]})$ from which, in the sequel we show that $S(\Sigma) \cong_\varepsilon \mathcal{S}(\{S^{\overline{\eta}}(\Sigma_{\Scc_k})\}_{k\in [1;\overline{N}]})$. 
Let $\Xi^{\overline{\eta}}_{k}$ be the set of states of $S^{\overline{\eta}}(\Sigma_{\Scc_k})$ for any $k\in [1;\overline{N}]$. Consider the relation $\mathcal{R}\subseteq ( \Xi_1 \times \Xi_2 \times ... \times \Xi_{\overline{N}}) \times 
( \Xi^{\overline{\eta}}_{1} \times \Xi^{\overline{\eta}}_{2} \times ... \times \Xi^{\overline{\eta}}_{\overline{N}})$ defined by $((\xi_1,\xi_2,...,\xi_{\overline{N}}),(\xi^{\prime}_1,\xi^{\prime}_2,...,\xi^{\prime}_{\overline{N}})) \in \mathcal{R}$ if and only if 
 $\overline{V}_k(\xi_k,\xi^{\prime}_k) \leq \underline{\alpha}^k(\varepsilon^k)$. Consider any $((\xi_1,\xi_2,...,\xi_{\overline{N}}),(\xi^{\prime}_1,\xi^{\prime}_2,...,\xi^{\prime}_{\overline{N}})) \in \mathcal{R}$. 
 We first note that by Algorithm \ref{alg}, $\varepsilon^k\leq\varepsilon$, $k\in [1;\overline{N}]$; hence, $\Vert (\xi_1,\xi_2,...,\xi_{\overline{N}}) - (\xi^{\prime}_1,\xi^{\prime}_2,...,\xi^{\prime}_{\overline{N}}) \Vert = \max_{k\in [1;\overline{N}]} \Vert \xi_{k}-\xi^{\prime}_k \Vert \leq 
 \max_{k\in [1;\overline{N}]} (\underline{\alpha}^k)^{-1}(\overline{V}_{k}(\xi_{k},\xi^{\prime}_k)) \leq 
 \max_{k\in [1;\overline{N}]} \varepsilon^k \leq \varepsilon$ from which, condition (i) of Definition \ref{ASR} holds. We now show that also condition (ii) holds. Consider any $\omega=(\omega_1,\omega_2,...,\omega_{\overline{N}})\in U^\ast$ and the transition $(\xi_1,\xi_2,...,\xi_{\overline{N}}) \rTo^{\omega}_{*} (\xi_{1,+},\xi_{2,+},...,\xi_{\overline{N},+})$ in system $S(\Sigma)$. 
 By Definition \ref{lab}, for any $k\in [1;\overline{N}]$, the transition $\xi_k \rTo^{ \nu_{k},\omega_{k}} \xi_{k,+}$ is in system $S(\Sigma_{\Scc_k})$, for an appropriate input label $\nu_{k}$. By definition of the $\Post$ operator, the inequalities in line 11 of Algorithm \ref{alg} coincide with the ones in (\ref{Sstatem}) and (\ref{Sstatem1}). Hence, by Theorem \ref{main2}, $S(\Sigma_{\Scc_k}) \cong_{\varepsilon^k} S^{\overline{\eta}}(\Sigma_{\Scc_k})$ from which, there exists a transition $\xi^{\prime}_k \rTo^{\nu^{\prime}_k,\omega^{\prime}_k} \xi^{\prime}_{k,+}$ in $S^{\overline{\eta}}(\Sigma_{\Scc_k})$ such that $\overline{V}_{k}(\xi_{k,+},\xi^{\prime}_{i,k}) \leq \underline{\alpha}^k(\varepsilon^k)$. We first note that by definition of $\mathcal{R}$, $((\xi_{1,+},\xi_{2,+}, ... ,\xi_{\overline{N},+}),(\xi^{\prime}_{1,+},\xi^{\prime}_{2,+},...,\xi^{\prime}_{\overline{N},+})) \in \mathcal{R}$. 
Secondly by Definition \ref{lab}, the transition $(\xi^{\prime}_1,\xi^{\prime}_2,...,\xi^{\prime}_{\overline{N}}) \rTo^{\omega^{\prime}} (\xi^{\prime}_{1,+},\xi^{\prime}_{2,+},...,\xi^{\prime}_{\overline{N},+})$, with $\omega^{\prime}=(\omega^{\prime}_1,\omega^{\prime}_2,...,\omega^{\prime}_{\overline{N}})$, is in $S^{\overline{\eta}}(\Sigma_{\Scc_k})$ from which, condition (ii) in Definition \ref{ASR} holds. Condition (iii) in Definition \ref{ASR} can be shown by using similar arguments. Finally, for any $\xi_k\in \Xi_k$ by choosing $\xi^{\prime}_k=[\xi_k]_{\overline{\eta}(k)}\in \Xi^{\overline{\eta}}_{k}$ we get $\overline{V}_{k}(\xi_k,\xi^{\prime}_k) \leq \overline{\alpha}^{k}(\Vert \xi_k - \xi^{\prime}_k \Vert) \leq \overline{\alpha}^{k}(\overline{\eta}(k))\leq \overline{\alpha}^{k}(\varepsilon^k)$. In particular, the first inequality holds by the inequality (i) in Theorem \ref{main2}, the second one by definition of operator $[\, .\, ]$ and the last one by Algorithm \ref{alg}. Hence, $\mathcal{R}(\Xi_1\times \Xi_2, \times ... \times \Xi_{\overline{N}})=\Xi^{\overline{\eta}}_{1}\times \Xi^{\overline{\eta}}_{2} \times ... \times \Xi^{\overline{\eta}}_{\overline{N}}$. Conversely, for any $\xi^{\prime}_k \in \Xi^{\overline{\eta}}_{k}$, by picking $\xi_k=\xi^{\prime}_k$ one gets $\overline{V}_{k}(\xi_k,\xi^{\prime}_k)=0\leq \underline{\alpha}^{k}(\varepsilon^k)$ from which, $\mathcal{R}^{-1}(\Xi^{\overline{\eta}}_{1}\times \Xi^{\overline{\eta}}_{2} \times ... \times \Xi^{\overline{\eta}}_{\overline{N}})=\Xi_1\times \Xi_2 \times ... \times \Xi_{\overline{N}}$, which concludes the proof. 
\end{proof}

\begin{figure}[!t]
\begin{center}
\includegraphics[scale=0.45]{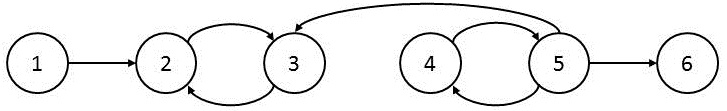}
\caption{Graph $\mathcal{G}=(\mathcal{V},\mathcal{E})$ associated with the network of control systems in the example reported in Section \ref{Sec:example}.}
\label{fig}
\end{center}
\end{figure}

\section{An academic example}\label{Sec:example}
Consider the network of control systems $\Sigma$ in (\ref{NetCS}) with $N=6$ and
\[
\begin{array}
{l}
f_{1}(x(t),u_{1}(t)) =  \kappa_{1,1} \frac{x_1(t)}{1+x_1^2(t)}+ u_1(t);
\\
f_{2}(x(t),u_{2}(t)) = \kappa_{2,1} \tanh(x_{2}(t)) + \kappa_{2,2} (\sech (x_{3}(t))-1) + x_{1}(t) ;
\\
f_{3}(x(t),u_{3}(t)) = \kappa_{3,1} x_3(t) + \kappa_{3,2} \sin(x_2(t))+ x_5(t)+u_3(t);
\\
f_{4}(x(t),u_{4}(t)) = \kappa_{4,1}(\cos(x_4(t))-1)+\kappa_{4,2}(\tanh(x_5(t)));
\\
f_{5}(x(t),u_{5}(t)) = \kappa_{5,1}\sin(x_5(t))+\kappa_{5,2}(\sech(x_4(t))-1)+u_{5}(t);
\\
f_{6}(x(t),u_{6}(t)) = \kappa_{6,1} \frac{x_6(t)}{1+\vert x_6(t)\vert}+x_5(t),
\end{array}
\]
where $x(t)=(x_{1}(t),x_{2}(t),...,x_{6}(t))$ for any $t \in \mathbb{N}_{0}$. 
We set $\kappa_{i,1}\in ]-1,1[$ for any $i\in[1;6]$, $\kappa_{i,2}\in \mathbb{R}$ for any $i\in[2;5]$, $\mathcal{X}_i=[-1,1]$ and $\mathcal{U}_i=[-1,1]$ for any $i\in[1;6]$. The goal is to construct a symbolic model of $\Sigma$ with accuracy $\varepsilon=0.01$. To this purpose we apply the results of the previous section. The resulting graph $\mathcal{G}=(\mathcal{V},\mathcal{E})$ is specified by $\mathcal{V}=[1;6]$ and 
$\mathcal{E}=\{(1,2),(2,3),(3,2),(4,5),(5,3),(5,4),(5,6)\}$ (see Fig. \ref{fig}). Strongly connected components of $\mathcal{G}$ are $\Scc_1$ with $\mathcal{V}_1 = \{1\}$,
$\Scc_2$ with $\mathcal{V}_2 = \{4,5\}$, $\Scc_3$ with $\mathcal{V}_3 = \{2,3\}$, and $\Scc_4$ with $\mathcal{V}_4 = \{6\}$. We are now ready to apply the three steps described in the previous section. Detailed calculations on this example are reported in \cite{Pola2014networkArxiv}. 
\\
\textit{Step \#1:} It is possible to show that $V_{i}:\mathbb{R}\times \mathbb{R} \rightarrow \mathbb{R}^{+}_{0}$, defined by  $V_i(x_i,x'_i)=|x_i - x'_i|$, $x_i,x'_i \in \mathbb{R}$, is a $\delta$--ISS Lyapunov function for subsystem $\Sigma_i$ for all $i\in [1;6]$. 
Hereafter, we only report detailed calculations for the case of $i=5$; the other cases follow analogously. By taking into account the Lipschitz property of the functions $z\to \sin(z)$, $z\to \sech(z)-1$, $z\in \mathbb{R}$, the following equalities/inequalities hold, for any $x_i, x_i'\in \mathbb{R}$, $i\in[1;6]$, $u_5, u_5' \in \mathbb{R}$:
\[
\begin{array}
{l}
V_5(f_5(x_1,x_2,\dots,x_6,u_5),f_5(x_1',x_2',\dots,x_6',u_5'))-V_5(x_5,x_5')=\nonumber \\ 
\vert f_5(x_1,x_2,\dots,x_6,u_5)-f_5(x_1',x_2',\dots,x_6',u_5') \vert - \vert x_5-x_5'\vert= \nonumber \\ 
\vert \kappa_{5,1}\sin(x_5)+\kappa_{5,2}(\sech(x_4)-1)+u_{5}- \kappa_{5,1}\sin(x_5')-\kappa_{5,2}(\sech(x_4')-1)-u_{5}'\vert -\vert x_5-x_5'\vert \le \nonumber \\ 
\vert \kappa_{5,1}\vert \vert x_5-x_5'\vert+\vert \kappa_{5,2}\vert \vert x_4-x_4'\vert +\vert u_5-u_5'\vert-\vert x_5-x_5' \vert\le \nonumber \\ 
-\left (1-\vert \kappa_{5,1}\vert \right )\vert x_5-x_5'\vert +\vert \kappa_{5,2}\vert \vert x_4-x_4'\vert +\vert u_5-u_5'\vert .
\end{array}
\]
The corresponding bounding constant and functions in Assumption (A1), are given by $L_5=2$ and for any $s\in \mathbb{R}_0^+$, $\overline \alpha_5(s)=\underline\alpha_5(s)=s$, $\rho_5(s)=\left (1-\vert \kappa_{5,1}\vert \right )s$, $\sigma_{5,4}(s)=\vert \kappa_{5,2}\vert s$, $\sigma_{5,j}(s)=0$, $j\in [1;6],\ j\ne 4$, $\sigma_5(s)=s$. By analogous computations we obtain for all subsystems, including $\Sigma_5$: 
$L_i=2$ for any $i\in [1;6]$ and, for any $s\in \mathbb{R}_0^+$, 
$\overline \alpha_i(s)=\underline \alpha_i(s)=s$, $\rho_i(s)=\left (1-\vert \kappa_{i,1}\vert \right )s$, $i\in[1;6]$, $\sigma_{i,j}(s)=0$, $i,j\in[1;6]$, $i\ne j$, $(i,j)\notin \left \{(2,1), (2,3), (3,2), (3,5), (4,5), (5,4), (6,5) \right \}$, $\sigma_{i,j}(s)=\vert \kappa_{i,2}\vert s$, $(i,j)\in \left \{(2,3), (3,2), (4,5), (5,4) \right \}$, $\sigma_{2,1}(s)=\sigma_{3,5}(s)=\sigma_{6,5}(s)=s$, $\sigma_i(s)=0$, $i=2,4,6$, $\sigma_1(s)=\sigma_3(s)=\sigma_5(s)=s$. Hence, Assumption (A1) is satisfied for any $i\in[1;6]$.\\
\textit{Step \#2:} We only need to apply Theorem \ref{main2} to strongly connected components $\Scc_2$ and $\Scc_3$ because $\Scc_1$ and $\Scc_4$ are composed each of a single control system. To this purpose it is readily seen that Assumption (A2) is verified for $g_i^{2}(s)=g_j^3(s)=s$, $s\in \mathbb{R}_0^+$, $i\in \{4,5\}$, $j\in \{2,3\}$. Moreover, 
$A_2=\diag\left( 1-\vert \kappa_{4,1}\vert,   1-\vert \kappa_{5,1}\vert\right)$, 
$A_3=\diag\left( 1-\vert \kappa_{2,1}\vert,   1-\vert \kappa_{3,1}\vert\right)$, 
$C_2 = [0,\vert \kappa_{4,2}\vert; \vert \kappa_{5,2}\vert ,0 ]$ and $C_3= [0,\vert \kappa_{2,2}\vert ; \vert \kappa_{3,2}\vert , 0 ]$. 
Condition $r(A^{-1}_k C_k)$ of Theorem \ref{main2} is satisfied for $k=2,3$, if and only if the small gain inequalities  
$\vert \kappa_{4,2}\kappa_{5,2}\vert / (\left (1-\vert \kappa_{4,1}\vert\right )\left (1-\vert \kappa_{5,1}\vert \right ))<1$ and 
$\vert \kappa_{2,2}\kappa_{3,2}\vert /(\left ( 1-\vert \kappa_{2,1}\vert \right )\left ( 1-\vert \kappa_{3,1}\vert\right ))<1$ hold.
For instance, the above inequalities are satisfied for $\kappa_{i,1}=0.5$, $\kappa_{i,2}=0.4$, $i=2,3,4,5$. 
Taking into account of the computations in \textit{Step \#1}, we now compute functions $\overline V_k$ and related constants and functions $L^k$, $\rho^k$, $\sigma^{k}_j$, $\sigma^k$, $k\in [1;4]$, $j\in [1;4]$, $j\ne k$.  For $k=1,4$, we have $\xi_1=x_1$, $\xi_1'=x_1'$, $\xi_4=x_6$, $\xi_4'=x_6'$, $\overline V_1(\xi_1, \xi_1')=V_1(x_1,x_1')$, $\overline V_4(\xi_4, \xi_4')=V_6(x_6,x_6')$. 
For $k=2,3$, we have 
$\xi_2=(x_4 ,x_5)$, $\xi_2'=(x_4' ,x_5')$, 
$\xi_3=(x_2 , x_3)$, 
$\xi_3'=(x_2' ,x_3')$, 
and
$\overline V_2(\xi_2,\xi_2')=\lambda_{i(1,2)}V_4(x_4,x_4')+\lambda_{i(2,2)}V_5(x_5,x_5')$, $\overline V_3(\xi_3,\xi_3')=\lambda_{i(1,3)}V_2(x_2,x_2')+\lambda_{i(2,3)}V_3(x_3,x_3')$. Thus, we have $L^2=2(\lambda_{i(1,2)}+\lambda_{i(2,2)})$, 
$L^3= 2(\lambda_{i(1,3)}+\lambda_{i(2,3)})$, $\underline \alpha^2(s)=\min\{\lambda_{i(1,2)},\lambda_{i(2,2)}\}s$, $\overline \alpha^2(s)=(\lambda_{i(1,2)}+\lambda_{i(2,2)})s$, $\underline \alpha^3(s)=\min\{\lambda_{i(1,3)},\lambda_{i(2,3)}\}s$, 
$\overline \alpha^3(s)=(\lambda_{i(1,3)}+\lambda_{i(2,3)})s$, 
$\sigma^2_j(s)=0$, $j=1,3,4$, $\sigma^2(s)=\lambda_{i(2,2)}s$, $\sigma^3_1(s)=\lambda_{i(1,3)}s$, $\sigma^3_2(s)=\lambda_{i(2,3)}s$, $\sigma^3(s)=\lambda_{i(2,3)}s$, $s\in \mathbb{R}_0^+$. Finally following \cite{SmallGainTh2}, let us choose $\underline\lambda_k$ such that each component of $\underline \lambda_k^T(A_k-C_k)>0$, $k=2,3$, and we can choose $\rho^2(s)=\left (\min\{\underline \lambda_2^T(A_2-C_2)\}/\max\{\underline \lambda_2\}\right )s$, $\rho^3(s)=\left (\min\{\underline\lambda_3^T(A_3-C_3)\}/\max\{\underline \lambda_3\}\right )s$, $s\in \mathbb{R}_0^+$. By the above choice of parameters $\kappa_{i,1}=0.5$, $\kappa_{i,2}=0.4$, $i\in[2;5]$, we 
can choose $\lambda_{i(1,2)}=11$, $\lambda_{i(2,2)}=13$, $\lambda_{i(1,3)}=1$ and $\lambda_{i(2,3)}=1$, by which we obtain 
$\rho^2(s)=0.0231s$ and $\rho^3(s)=0.1s$, $s\in \mathbb{R}_0^+$. \\
\textit{Step \#3:} We now apply Algorithm \ref{alg} to design the vector of quantization parameters $\overline{\eta}\in \mathbb{R}^{+}_{4}$. 
The leaf of the DAG associated with $\mathcal{G}$ is $\Scc_{4}$. Since $\overline{\eta}(4)=\infty$, condition in line 5 is satisfied and $\varepsilon^4$ is updated in line 7 to $\varepsilon=0.01$. Parameters $\overline{\eta}(4)$ and $\eta^{\ast}_2$ are chosen as $1.66\cdot 10^{-3}$ in line 11. The set $\SCC_\temp$ is updated in line 14 to $\{\Scc_{1},\Scc_{2},\Scc_{3}\}$. 
The leaf of the resulting $\SCC_\temp$ is now $\Scc_3$. Since $\overline{\eta}(3)=\infty$, condition in line 5 is satisfied and $\varepsilon^3$ is updated in line 9 to $\overline{\eta}(4)=1.66\cdot 10^{-3}$. Parameters $\overline{\eta}(3)$, $\eta^{\ast}_1$ and $\eta^{\ast}_2$ are chosen as $2.38 \cdot 10^{-5}$ in line 11. The set $\SCC_\temp$ is updated in line 14 to $\{\Scc_{1},\Scc_{2}\}$. 
The leaves of the resulting $\SCC_\temp$ are now $\Scc_1$ and $\Scc_2$. Let us start by processing $\Scc_1$. Since $\overline{\eta}(1)=\infty$, condition in line 5 is satisfied and $\varepsilon^1$ is updated in line 9 to $\overline{\eta}(3)=2.38 \cdot 10^{-5}$. Parameter $\overline{\eta}(1)$ is chosen as $3.96 \cdot 10^{-6}$ in line 11. Consider now $\Scc_2$. Since $\overline{\eta}(2)=\infty$, condition in line 5 is satisfied and $\varepsilon^2$ is updated in line 9 to $\min\{\overline{\eta}(3),\overline{\eta}(4)\}=2.38 \cdot 10^{-5}$. Parameter $\overline{\eta}(2)$ is chosen as $9.91 \cdot 10^{-8}$ in line 11. Set $\SCC_\temp$ is updated in line 14 to the empty set and the algorithm is over. 
We finally obtain 
$\overline{\eta}=(3.96 \cdot 10^{-6},9.91 \cdot 10^{-8},2.38 \cdot 10^{-5},1.66\cdot 10^{-3})$ 
and consequently, 
$\eta=(3.96 \cdot 10^{-6},2.38 \cdot 10^{-5},2.38 \cdot 10^{-5},9.91 \cdot 10^{-8}, 9.91 \cdot 10^{-8},1.66\cdot 10^{-3})$. \\
We conclude this section by performing a complexity analysis. By a straightforward computation, space and time complexity in computing the collection of symbolic models $S^{\eta}(\Sigma_i)$ with $i\in [1;6]$ are given by 
$\Sigma_{i\in [1;6]} \scomplex(S^{\eta}(\Sigma_i))=1.68 \cdot 10^{29}$ and $\Sigma_{i\in [1;6]} \tcomplex(S^{\eta}(\Sigma_i))=2.02\cdot 10^{22}$, respectively. 
The space and time complexity in constructing the composition $\mathcal{S}(\{S^{\eta}(\Sigma_i)\}_{i\in [1;6]})$ are given by $ \scomplex (\mathcal{S}(\{S^{\eta}(\Sigma_i)\}_{i\in [1;6]}))= 
5.31 \cdot 10^{99}$ and $\tcomplex (\mathcal{S}(\{S^{\eta}(\Sigma_i)\}_{i\in [1;6]}))= 3.04\cdot 10^{66}$. 
We now compare the above computational complexity with the computational complexity arising when applying the discrete--time version of the results reported in \cite{PolaAutom2008}. 
To this purpose we consider $\Sigma$ as a monolithic control system and apply Proposition \ref{Th1}, which corresponds to Theorem 5.1 of \cite{PolaAutom2008} in the discrete--time domain. It is possible to show that function $V^{\ast}$ defined by $V^{\ast}((x_1,x_2,...,x_6),(x'_1,x'_2,...,x'_6))=\sum_{i\in [1;6]} \lambda^{\ast}(i)\vert x_i -x'_i \vert$, with $\lambda^{\ast}=(3,1,1,11,13,1)$ 
(note that $\lambda_i$, $i\in [2;5]$ are the same used above for subsystems $\Scc_2$, $\Scc_3$), 
is a $\delta$--ISS Lyapunov function for $\Sigma$. Corresponding bounding constant and functions associated with $V^\ast$ are given by 
$L^{\ast}=60$, $\underline{\alpha}^{\ast}(s)=s$, $\overline{\alpha}^{\ast}(s)=30s$, $\rho^{\ast}(s)=7.7\cdot 10^{-3}\, s$, $\sigma^{\ast}(s)=17s$, $s\in\mathbb{R}^{+}_{0}$. By applying Proposition \ref{Th1} to the entire control system $\Sigma$, the (uniform) quantization parameter $\eta^{\ast}$ obtained is upper bounded by $10^{-6}$; we set $\eta^{\ast}=10^{-6}$. The corresponding space and time complexity in computing the symbolic model associated with $S(\Sigma)$, denoted $S^{\eta^{\ast}}(\Sigma)$, are given by $\scomplex(S^{\eta^{\ast}}(\Sigma))=2.62\cdot 10^{113}$ and $\tcomplex(S^{\eta^{\ast}}(\Sigma))=4.11 \cdot 10^{75}$. 

\section{Conclusions} \label{Sec:discussion}
In this paper we proposed networks of symbolic models that approximate networks of discrete--time nonlinear control systems in the sense of approximate bisimulation for any desired accuracy. 
In future work we plan to extend the results of this paper to continuous--time nonlinear control systems. The extension is not straightforward because it requires appropriate techniques to find finite approximations of trajectories of continuous--time control systems; in this regard, spline based approximation schemes proposed in \cite{PolaSCL10} and \cite{BorriIJC2012} can be of help.

\bibliographystyle{plain}
\bibliography{biblio2}

\section{Appendix}

\subsection{Notation}
The symbol $\card(X)$ indicates the cardinality of a finite set $X$. Given a pair of sets $X$ and $Y$ and a relation $\mathcal{R}\subseteq X\times Y$, the symbol $\mathcal{R}^{-1}$ denotes the inverse relation of $\mathcal{R}$, i.e.
$\mathcal{R}^{-1}=\{(y,x)\in Y\times X:( x,y)\in \mathcal{R}\}$. We denote $\mathcal{R}(X)=\{y\in Y | \exists x\in X \text{ s.t. } (x,y)\in \mathcal{R}\}$ and $\mathcal{R}^{-1}(Y)=\{x\in X | \exists y\in Y \text{ s.t. }  (x,y)\in \mathcal{R}\}$. The symbols $\mathbb{N}_0$, $\mathbb{Z}$, $\mathbb{R}$, $\mathbb{R}^{+}$ and $\mathbb{R}_{0}^{+}$ denote the set of nonnegative integer, integer, real, positive real, and nonnegative real numbers, respectively. The symbol $\mathbb{R}^{+}_{n}$ denotes the positive orthant of $\mathbb{R}^{n}$. Given $n\in \mathbb{N}_0$ and $n>0$ we denote by $[1;n]$ the set $\{1,2,...,n\}$. Given $a_1,a_2,...,a_n\in\mathbb{R}$, the symbol $\diag(a_1,a_2,...,a_n)$ denotes the diagonal matrix whose entries in the diagonal are $a_i$. 
For a matrix $A=(a_{ij})_{i,j\in [1;n]}$, the inequality $A>0$ (resp. $A<0$) is meant component-wise, i.e. $a_{ij}>0$ (resp. $a_{ij}<0$) for all $i,j\in [1;n]$.
The symbol $r(A)$ denotes the spectral radius of a square matrix $A$, i.e. $r(A)=\max_{i=1,2,...,n}|\lambda_i|$, where $\lambda_i$, $i=1,2,...,n$, are the eigenvalues of $A$.  
Given $a\in \mathbb{R}$, the symbol $|a|$ denotes the absolute value of $a$ and $\lceil a \rceil$ the ceiling of $a$, i.e. $\lceil a \rceil=\min\{{n\in\mathbb{Z} \vert n\geq a}\}$. 
Given a vector $x\in\mathbb{R}^{n}$ we denote by $x(i)$ the $i$--th element of $x$ and by $\Vert x\Vert$ the infinity norm of $x$. Given $a\in\mathbb{R}$ and $\Omega\subseteq \mathbb{R}^{n}$ the symbol $a\,\Omega$ denotes the set $\{y\in\mathbb{R}^{n}| \exists (\omega_1,\omega_2,...,\omega_n)\in \Omega \text{ s.t. } y=(a\omega_1,a\omega_2,...,a\omega_n)\}$. 
The identity function is denoted by $\Id$. A continuous function \mbox{$\gamma:\mathbb{R}_{0}^{+}\rightarrow\mathbb{R}_{0}^{+}$} is said to belong to class $\mathcal{K}$ if it is strictly increasing and \mbox{$\gamma(0)=0$}; function $\gamma$ is said to belong to class $\mathcal{K}_{\infty}$ if \mbox{$\gamma\in\mathcal{K}$} and $\gamma(r)\rightarrow\infty$ as $r\rightarrow\infty$. 
Given $\eta\in\mathbb{R}^{+}$ and $X\subseteq \mathbb{R}^{n}$, we set $[X]_{\eta}=(\eta\,\mathbb{Z}^{n}) \cap X$; if $X$ is convex and with interior there always exists $\eta\in\mathbb{R}^{+}$ such that for any $x\in X$ there exists $y\in [X]_{\eta}$ such that $\Vert x-y  \Vert \leq \eta$. Given $x=(x_1,x_2,...,x_n)\in  \mathbb{R}^{n}$ and $\eta\in\mathbb{R}^{+}$, define $[x]_{\eta}=(\eta\lceil x_1/\eta \rceil,\eta\lceil x_2/\eta \rceil,...,\eta\lceil x_n/\eta \rceil)\in \eta \mathbb{Z}^{n}$; note that $\Vert x- [x]_{\eta} \Vert \leq \eta$. 
A directed graph $\mathcal{G}$ is specified by a pair $(\mathcal{V},\mathcal{E})$ where $\mathcal{V}$ is the set of vertices and $\mathcal{E}\subseteq \mathcal{V} \times \mathcal{V}$ is the set of edges. 
A pair $(\mathcal{V}',\mathcal{E}')$ is a subgraph of $\mathcal{G}=(\mathcal{V},\mathcal{E})$ if $\mathcal{V}' \subset \mathcal{V}$ and $\mathcal{E}' \subset \mathcal{E}$. 
Strongly connected components of a directed graph $\mathcal{G}$ are its maximal strongly connected subgraphs.

\subsection{Systems, Composition and Approximate Equivalence}\label{sec:ApproxEquiv}
We start by introducing the notion of systems that we use as a unified mathematical paradigm to describe nonlinear control systems and their symbolic models.

\begin{definition}
\label{systems}
\cite{paulo}
A system is a quintuple $S=(X,U,\rTo,Y,H)$, consisting of a set of states $X$, a set of inputs $U$, a transition relation $\rTo \subseteq X\times U\times X$, a set of outputs $Y$ and an output function $H:X\rightarrow Y$.
\end{definition}

A transition $(x,u,x^{\prime})\in\rTo$ of $S$ is denoted by $x\rTo^{u}x^{\prime}$.  
System $S$ is said to be \textit{symbolic} if $X$ and $U$ are finite sets and \textit{metric} if the output set $Y$ is equipped with a metric $\mathbf{d}:Y\times Y\rightarrow\mathbb{R}_{0}^{+}$. Composition of systems in formalized hereafter.

\begin{definition}
\label{lab}
Given a collection of systems $S_{i} = (X_{i} , X_1 \times ... \times X_{i-1} \times X_{i+1} \times ... \times X_{N} \times U_{i},\rTo_{i},Y_{i},H_{i})$, ($i\in [1;N]$), define the system $\mathcal{S}(\{S_i\}_{i\in [1;N]})= (X,U,\rTo,Y,H)$ where $X=X_{1}\times X_{2}\times ... \times X_{N}$, $U=U_{1}\times U_{2}\times ... \times U_{N}$,  $(x_{1},x_{2},...,x_{N})$ $\rTo^{(u_1,u_2,...,u_N)} (x_{1}^{+},x_{2}^{+},...,x_{N}^{+})$ if $x_{i} \rTo_{i}^{(x_1,...,x_{i-1},x_{i+1},...,x_N,u_i)} x_{i}^{+}$ for any $i\in [1;N]$, $Y=X_{1}\times X_{2} \times ... \times X_{N}$ and $H(x)=x$. 
\end{definition}

In the above definition, note that if systems $S_i$ are equipped with metric $\mathbf{d}_{i}$ then system $\mathcal{S}(\{S_i\}_{i\in [1;N]})$ is equipped with metric $\mathbf{d}((x_1,x_2,...,x_N),(x'_1,x'_2,...,x'_N))=\max_{i\in [1;N]}\mathbf{d}_{i}(x_{i},x'_{i})$. We conclude this section by recalling the notion of approximate bisimulation.

\begin{definition}
\cite{AB-TAC07}
\label{ASR}
Let $S^{i}=(X^{i},U^{i},\rTo_{i},Y^{i},H^{i})$ ($i=1,2$) be metric systems with the same output sets $Y^{1}=Y^{2}$ and metric $\mathbf{d}$, and let $\varepsilon\in\mathbb{R}^{+}_{0}$ be a given precision. A relation $\mathcal{R}\subseteq X^{1}\times X^{2}$ 
is an $\varepsilon$--approximate bisimulation relation if for all $(x^1,x^2)\in \mathcal{R}$ the following conditions are satisfied: 
(i) $\mathbf{d}(H^{1}(x^{1}),H^{2}(x^{2}))\leq\varepsilon$; 
(ii) For any $x^{1}\rTo_{1}^{u^{1}}x^{1}_{+}$ there exists $x^{2}\rTo_{2}^{u^{2}}x^{2}_{+}$ such that $(x^{1}_{+},x^{2}_{+})\in \mathcal{R}$;
(iii) For any $x^{2}\rTo_{2}^{u^{2}}x^{2}_{+}$ there exists $x^{1}\rTo_{1}^{u^{1}}x^{1}_{+}$ such that $(x^{1}_{+},x^{2}_{+})\in \mathcal{R}$.
Systems $S^1$ and $S^2$ are approximately bisimilar with precision $\varepsilon$, denoted by $S^1 \cong_{\varepsilon} S^2$, if there exists an $\varepsilon$--approximate bisimulation relation $\mathcal{R}$ between $S^{1}$ and $S^{2}$ such that $\mathcal{R}(X^{1})=X^{2}$ and  $\mathcal{R}^{-1}(X^{2})=X^{1}$. 
\end{definition}

\end{document}